\documentclass[a4paper,11pt]{article}
\usepackage{graphicx,epsfig,picins}
\usepackage{fancyhdr,fancybox}
\usepackage{indentfirst}
\usepackage{titlesec}
\usepackage{verbatim}
\usepackage[sort&compress, numbers]{natbib}
\usepackage{array,dcolumn,tabularx}
\usepackage{setspace}
\usepackage{geometry}
\usepackage{extarrows,chemarrow,xypic} 
\usepackage[small]{caption2}
\usepackage{sectsty}
\usepackage{microtype}
\DisableLigatures[f]{encoding = *, family = * }

\usepackage{times}

\usepackage{latexsym}
\usepackage{amsmath}
\usepackage{amssymb}
\usepackage{amsbsy}
\usepackage{amsthm}
\usepackage{amsfonts}
\usepackage{mathrsfs}
\usepackage{bm}
\usepackage{relsize}
\usepackage{hyperref}
\usepackage{color}

\newcommand{\paperfont}{\fontsize{11pt}{1.2\baselineskip}\selectfont}
\begin{document}
	
\theoremstyle{definition}
\makeatletter
\thm@headfont{\bf}
\makeatother
\newtheorem{theorem}{Theorem}[section]
\newtheorem{definition}[theorem]{Definition}
\newtheorem{lemma}[theorem]{Lemma}
\newtheorem{proposition}[theorem]{Proposition}
\newtheorem{corollary}[theorem]{Corollary}
\newtheorem{remark}[theorem]{Remark}
\newtheorem{example}[theorem]{Example}
\newtheorem{assumption}[theorem]{Assumption}

\lhead{}
\rhead{}
\lfoot{}
\rfoot{}

\renewcommand{\refname}{References}
\renewcommand{\figurename}{Figure}
\renewcommand{\tablename}{Table}
\renewcommand{\proofname}{Proof}

\newcommand{\dnumiag}{\mathrm{diag}}
\newcommand{\tr}{\mathrm{tr}}
\newcommand{\dnum}{\mathrm{d}}
\newcommand{\Enum}{\mathbb{E}}
\newcommand{\Pnum}{\mathbb{P}}
\newcommand{\Rnum}{\mathbb{R}}
\newcommand{\Cnum}{\mathbb{C}}
\newcommand{\Znum}{\mathbb{Z}}
\newcommand{\Nnum}{\mathbb{N}}
\newcommand{\RomanNum}[1]{\uppercase\expandafter{\romannumeral #1\relax}}
\newcommand{\abs}[1]{\left\vert#1\right\vert}
\newcommand{\set}[1]{\left\{#1\right\}}
\newcommand{\norm}[1]{\left\Vert#1\right\Vert}
\newcommand{\innp}[1]{\langle {#1}]}
\newcommand\mi{\mathrm{i}}
\newcommand\dif{\,\mathrm{d}}
\newcommand\Pro{\mathbb{P}}
\newcommand\outercontrollable{\textit{controllable}}
\newcommand\innercontrollable{\textit{Inner-controllable}}
\renewcommand{\theequation}{\arabic{section}.\arabic{equation}}

\title{\textbf{Moderate and $L^p$ maximal inequalities for diffusion processes and conformal martingales}}
\author{Xian Chen$^1$,\;\;\;Yong Chen$^2$,\;\;\;Yumin Cheng$^2$,\;\;\;Chen Jia$^{3,*}$ \\
\footnotesize $^1$School of Mathematics Sciences, Xiamen University, Xiamen 361005, Fujian, China \\
\footnotesize $^2$School of Mathematics and Statistics, Jiangxi Normal University, Nanchang 330022, Jiangxi, China \\
\footnotesize $^3$Applied and Computational Mathematics Division, Beijing Computational Science Research Center, Beijing 100193, China \\
\footnotesize Correspondence: chenjia@csrc.ac.cn}

\date{}
\maketitle
\thispagestyle{empty}

\paperfont
\titleformat*{\section}{\large\bfseries}

\begin{abstract}
The $L^p$ maximal inequalities for martingales are one of the classical results in the theory of stochastic processes. Here we establish the sharp moderate maximal inequalities for one-dimensional diffusion processes, which include the $L^p$ maximal inequalities as special cases. Moreover, we apply our theory to many specific examples, including the Ornstein-Uhlenbeck (OU) process, Brownian motion with drift, reflected Brownian motion with drift, Cox-Ingersoll-Ross process, radial OU process, and Bessel process. The results are further applied to establish the moderate maximal inequalities for some high-dimensional processes, including the complex OU process and general conformal local martingales.\\


\end{abstract}

\section{Introduction}
The moderate and $L^p$ maximal inequalities for continuous martingales are one of the classical results in probability theory. Let $M = (M_t)_{t\geq 0}$ be a continuous local martingale with vanishing at zero. The Burkholder-Davis-Gundy (BDG) inequality \cite[Chapter IV, Exercise 4.25]{revuz1999continuous} claims that for any moderate function $F$ (see Definition \ref{moderate} below), there exist two constants $c_F,C_F>0$ such that for any stopping time $\tau$ of $M$,
\begin{equation*}
c_F\Enum F([M,M]_\tau) \leq \Enum\left[\sup_{0\leq t\leq\tau}F(M_t)\right] \leq C_FF([M,M]_\tau).
\end{equation*}
where $[M,M] = ([M,M]_t)_{t\geq 0}$ is the quadratic variation process of $M$. This type of inequalities is referred to as moderate maximal inequalities since it holds for any moderate function. Since $F(x) = x^p$ is a moderate function for any $p>0$, there exist two constants $c_p,C_p>0$ such that for any stopping time $\tau$ of $M$,
\begin{equation*}
c_p\Enum[M]^{p/2}_\tau \leq \Enum\left[\sup_{0\leq t\leq\tau}|M_t|^p\right] \leq C_p\Enum[M]^{p/2}_\tau,
\end{equation*}
which gives the BDG-type $L^p$ maximal inequalities.

Over the past two decades, significant progress has been made in the $L^p$ maximal inequalities for diffusion processes \cite{dubins1994optimal, graversen1998maximal, graversen1998optimal, graversen2000maximal, peskir2001bounding, botnikov2006davis, lyulko2014sharp, yan2004ratio, yan2004maximal, yan2005lp, yan2005lpestimates, chen2017identification, shen2019some, jia2020moderate} and continuous-time Markov chains \cite{jia2019sharp}. In particular, Peskir \cite{peskir2001bounding} have proved the $L^1$ maximal inequalities for a wide class of one-dimensional diffusions using the Lenglart domination principle. Subsequently, the $L^1$ maximal inequalities have been generalized to the $L^p$ case for some special diffusions and special ranges of $p$ \cite{yan2004ratio, yan2004maximal, yan2005lp, yan2005lpestimates, chen2017identification, shen2019some}. However, thus far, very few results have been obtained about the moderate maximal inequalities for diffusions \cite{jia2020moderate}, which include the $L^p$ maximal inequalities as special cases. This is because the majority of existing results based their proof on the application of the Lenglart domination principle, which performs well in the $L^p$ case but fails in the more general moderate case. Recently, Jia and Zhao \cite{jia2020moderate} have established the moderate maximal inequalities for the Ornstein-Uhlenbeck (OU) process. Specifically, let $X = (X_t)_{t\geq 0}$ be an OU process solving the stochastic differential equation
\begin{equation}\label{OU}
\dif X_t = -\alpha X_t\dif t+\dif B_t,\;\;\;X_0 = 0,
\end{equation}
where $\alpha>0$. Then for any moderate function $F$, there exist two constants $c_{\alpha,F},C_{\alpha,F}>0$ such that for any stopping time of $X$,
\begin{equation}\label{eqou}
c_{\alpha,F}\Enum F\bigl(\log^{1/2}(1+\alpha\tau)\bigr)\leq \Enum F(X^*_\tau)\leq C_{\alpha,F}\Enum F\bigl(\log^{1/2}(1+\alpha\tau)\bigr).
\end{equation}
The aim of the present paper is to generalize the above result and establish the moderate maximal inequalities for a wide class of one-dimensional diffusions and even higher-dimensional processes. Our method is based on the ``good $\lambda$ inequality"  introduced by Burkholder \cite{burkholder1973distribution} and is different from the previous method based on the Lenglart domination principle. Once the moderate maximal inequalities have been developed, the $L^p$ maximal inequalities follow naturally for any $p>0$.

The structure of this paper is organized as follows. In Section 2, we present the general theory of the moderate and $L^p$ maximal inequalities for one-dimensional diffusions. Section 3 is devoted to the proof of the main theorems. In Sections 3-6, we apply our theory to some specific examples and establish their moderate maximal inequalities; these examples include the OU process, Brownian motion with drift, reflected Brownian motion with drift, Cox-Ingersoll-Ross process, radial OU process, and Bessel process. In Sections 7 and 8, our theory is further applied to establish a novel type of maximal inequalities for some two-dimensional stochastic processes including the complex OU process, complex Brownian motion, and general conformal local martingales.

\section{Moderate maximal inequalities for diffusions}\label{sec2}
Let $X = (X_t)_{t\geq 0}$ be a one-dimensional time-homogeneous diffusion process starting from zero, which is the (weak) solution to the stochastic differential equation (SDE)
\begin{equation}\label{model}
\dif X_t = b(X_t)\dif t+\sigma(X_t)\dif B_t,\;\;\;X_0 = 0,
\end{equation}
where $b:\Rnum\rightarrow\Rnum$ is Borel measurable, $\sigma:\Rnum\rightarrow[0,\infty)$ is locally bounded, and $B = (B_t)_{t\geq 0}$ is a standard Brownian motion defined on some filtered probability space $(\Omega,\mathcal{F},\{\mathcal{F}_t\},P)$ satisfying the usual conditions. Let $X^*$ denote the maximum process of $\abs{X}$ defined by
\begin{equation*}
X^*_t=\sup_{0\leq s\leq t} \abs{X_s}.
\end{equation*}
Recall that the generator of $X$ is defined by
\begin{equation*}
\mathcal{L} = b(x)\frac{d}{dx}+\frac{1}{2}\sigma^2(x)\frac{d^2}{dx^2}.
\end{equation*}

Let $\mathbb{R}_+ = [0,\infty)$. Before stating our main results, we recall the following definition of moderate functions \cite[Page 164]{revuz1999continuous}.

\begin{definition}\label{moderate}
A function $F: \mathbb{R}_+ \rightarrow \mathbb{R}_+$ is called \emph{moderate} if\\
(a) it is a continuous increasing function vanishing at zero,\\
(b) there exists $\beta>1$ such that
\begin{equation}\label{requirement}
\sup_{x>0}\frac{F(\beta x)}{F(x)} <\infty.
\end{equation}
\end{definition}

In the above equation, we stipulate that $0/0 = 1$ and $x/0 = \infty$ for any $x>0$.
It is easy to see that if $F$ is moderate, then \eqref{requirement} holds for any $\beta\geq 1$ \cite[Page 164]{revuz1999continuous}. In particular, $F(x) = x^p$ is a moderate function for any $p>0$.

We next introduce the concept of controllable processes, which extends the definition given in \cite[Definition 3.2]{jia2019sharp}.

\begin{definition}\label{outercontrollable}
The process $X$ is called \emph{controllable} if there exist constants $\gamma,C>0$ and $\beta>1$ such that for any $t\geq 0$ and $\lambda>0$,
\begin{equation}\label{outereq}
\sup_{\abs{x}=\lambda}\Pro_x\bigl(X^*_t\geq\beta\lambda\bigr)\leq C\Pro_0\bigl(X^*_t\geq\gamma\lambda\bigr),
\end{equation}
where $\Pro_x(\cdot) = \Pro(\cdot|X_0 = x)$.
\end{definition}

The following two theorems, whose proof can be found in Section \ref{proof}, give the upper and lower bounds of the moderate maximal inequalities for diffusions.

\begin{theorem}\label{upperbound}
Let $g:\mathbb{R}_+\to \mathbb{R}_+$ be a strictly increasing continuous function with $g(0) = 0$. Suppose that $X$ is controllable and there exist constants $p>0$ and $C_p>0$ such that the following $L^p$ maximal inequality holds for any $t\geq 0$:
\begin{equation}\label{eqlp}
\Enum (X^{*}_t)^p\leq C_p\Enum (g(t))^p.
\end{equation}
Then for any moderate function $F$, there exists a constant $C_F>0$ such that for any stopping time $\tau$ wth respect to the filtration $\{\mathcal{F}_t\}$,
\begin{equation}\label{uppereq}
\Enum F\bigl(X^*_\tau\bigr)\leq C_F\Enum F\bigl(g(\tau)\bigr).
\end{equation}
\end{theorem}

\begin{theorem}\label{lowerbound}
Let $g:\mathbb{R}_+\to \mathbb{R}_+$ be a strictly increasing continuous function with $g(0) = 0$ and let $f:\Rnum\rightarrow\Rnum$ be a $C^2$ function satisfying $\mathcal{L}f = 1$ and $f(0)=0$. Suppose that exists $\beta >1$ such that the following condition holds:
\begin{equation}\label{ieq_lo}
\lim_{\delta\downarrow 0}
\sup_{\substack{\lambda>0,\atop \abs{x}<\delta\lambda, \abs{y}<\delta\lambda}}\frac{f(y)-f(x)}
{g^{-1}(\beta\lambda)-g^{-1}(\lambda)} = 0.
\end{equation}
Then for any moderate function $F$, there exists a constant $c_F>0$ such that for any stopping time $\tau$ wth respect to the filtration $\{\mathcal{F}_t\}$,
\begin{equation}
\Enum F\bigl(X^*_\tau\bigr) \geq c_F\Enum F\bigl(g(\tau)\bigr).
\label{lowereq}
\end{equation}
\end{theorem}

Combining the above two theorems, we obtain the following corollary, which is the main result of this paper.
\begin{corollary}\label{cor2}
Suppose that the conditions of both Theorems \ref{upperbound} and \ref{lowerbound} are satisfied. Then for any moderate function $F$, there exist two constants $c_F,C_F>0$ such that the following moderate maximal inequalities hold for any stopping time $\tau$ wth respect to the filtration $\{\mathcal{F}_t\}$:
\begin{equation*}
c_F\Enum F\bigl(g(\tau)\bigr) \leq \Enum F\bigl(X^*_\tau\bigr) \leq C_F\Enum F\bigl(g(\tau)\bigr).
\end{equation*}
In particular, for any $p>0$, there exist two constants $c_p,C_p>0$ such that the following $L^p$ maximal inequalities hold for any stopping time $\tau$ wth respect to the filtration $\{\mathcal{F}_t\}$:
\begin{equation*}
c_p\Enum (g(\tau))^p \leq \Enum (X^*_\tau)^p \leq C_p\Enum (g(\tau))^p.
\end{equation*}
\end{corollary}

\section{Proof of the main theorems}\label{proof}
Here we shall give the proof of Theorems \ref{upperbound} and \ref{lowerbound} using the following classical results, whose proof can be found in \cite[Chapter IV, Lemma 4.9]{revuz1999continuous}.

\begin{lemma}\label{Lemma3.3}
Let $X$ and $Y$ be two nonnegative random variables. Let ${\phi\colon{\Rnum}_+\rightarrow{\Rnum}_+} $ be a function satisfying ${\phi(\delta)\rightarrow 0}$ as ${\delta\rightarrow 0}$. Suppose that there exists $\beta>1$ such that the following good $\lambda$ inequality holds for any $\delta,\lambda>0$:
\begin{equation*}\label{eq3.3.1}
\Pro(X\geq \beta\lambda, Y < \delta\lambda) \leq \phi(\delta)\Pro(X \geq \lambda).
\end{equation*}
Then for any moderate function $F$, there exists a positive constant $C$ depending on $F$, $\beta$, and $\phi$ such that
\begin{equation*}\label{eq3.3.2}
\Enum F(X) \leq C\Enum F(Y).
\end{equation*}
\end{lemma}

We are now in a position to prove Theorem \ref{upperbound}.

\begin{proof}[\textbf{Proof of Theorem \ref{upperbound}}]\label{proof_upper}
For any $x>0$, let $\tau_x = \inf\{t\geq 0: \abs{X_t}\geq x\}$. Since $X$ is controllable, there exists $\beta>1$ and $\gamma, C,\lambda>0$ such that for any $s>0$,
\begin{equation*}
\sup_{\abs{x}=\lambda}\Pro_x(X^*_s\geq \beta\lambda) \leq C\Pro_0(X^*_s\geq \gamma\lambda).
\end{equation*}
Then for any $\delta>0$, it is easy to check that
\begin{equation*}
\Pro_0\bigl(X^*_\tau\geq \beta\lambda,g(\tau)<\delta\lambda\bigr)
\leq \Pro_0(X^*_{s\vee{\tau_\lambda}}\geq \beta\lambda, \tau >\tau_\lambda),
\end{equation*}
where $s=g^{-1}(\delta\lambda)$. By the strong Markov property of $X$, we have
\begin{align*}
\Pro_0\bigl(X^*_\tau \geq \beta\lambda,g(\tau)<\delta\lambda\bigr)
&\leq \Enum_0\bigl[1_{\{\tau >\tau_\lambda\}}\Pro_0(X^*_{\tau_\lambda+s}\geq \beta\lambda|\mathscr{F}_{\tau_{\lambda}})\bigr]\\
&= \Enum_0\bigl[1_{\{\tau>\tau_\lambda\}}
\Pro_{X_{\tau_{\lambda}}}(X^*_s\geq \beta\lambda)\bigr]\\
&\leq \sup_{\abs{x}=\lambda}\Pro_x(X^*_s\geq \beta\lambda)\Pro_0(\tau >\tau_\lambda)
\leq C\Pro_0(X^*_s\geq \gamma\lambda)\Pro_0(\tau >\tau_\lambda).
\end{align*}
It then follows from Chebyshev's inequality that
\begin{equation*}
\begin{split}
\Pro_0(X^*_s\geq \gamma\lambda)
&\leq \frac{\Enum (X^{*}_s)^p}{(\gamma\lambda)^p} \leq \frac{C_p(g(s))^p}{(\gamma\lambda)^p}
= \frac{C_p(\delta\lambda)^p}{(\gamma\lambda)^p} = \frac{C_p\delta^p}{\gamma^p},
\end{split}
\end{equation*}
where $C_p$ is the constant in \eqref{eqlp}. This shows that
\begin{equation*}
\Pro_0\bigl(X^*_\tau\geq \beta\lambda, g(\tau)< \delta\lambda\bigr)
\leq \frac{CC_p\delta^p}{\gamma^p}\Pro_0(X^*_\tau\geq \lambda).
\end{equation*}
The desired result then follows from Lemma \ref{Lemma3.3}.
\end{proof}


We next give the proof of Theorem \ref{lowerbound}.

\begin{proof}[\textbf{Proof of Theorem \ref{lowerbound}}]\label{proof_lower}
For any $\delta>0$, it is easy to check that
\begin{equation*}
\Pro_0\bigl(g(\tau)\geq \beta\lambda, X^*_\tau <\delta\lambda\bigr)\leq \Pro_0(\tau \geq r, X^*_s <\delta\lambda),
\end{equation*}
where $r=g^{-1}(\lambda)$ and $s=g^{-1}(\beta\lambda)$. By the Markov property of $X$, we have
\begin{align*}
\Pro_0(g(\tau)\geq \beta\lambda, X^*_\tau <\delta\lambda)&\leq\Enum_0\bigl[1_{\{\tau\geq r\}}\Pro_0(X^*_s <\delta\lambda\vert \mathscr{F}_r)\bigr]\\
&\leq\Enum_0\bigl[1_{\{\tau\geq r\}}\Pro_{X_{r}}(X^*_{s-r} <\delta\lambda)\bigr]\\
&\leq\sup_{\abs{x}<\delta\lambda}\Pro_{x}(X^*_{s-r}<\delta\lambda)\Pro_0(\tau\geq r).
\end{align*}
For any $x>0$, let $\tau_x = \inf\{t\geq 0: \abs{X_t}\geq x\}$. Applying It\^{o}'s formula, we obtain
\begin{equation*}
f\bigl(X_{\tau_{\delta\lambda}\wedge t}\bigr)
= f\bigl(X_0\bigr)+\int_0^{\tau_{\delta\lambda}\wedge t}\mathcal{L} f\bigl(X_s^x\bigr)\dif s
+\int_0^{\tau_{\delta\lambda}\wedge t} f^{\prime}\bigl(X_s\bigr)\sigma\bigl(X_s\bigr)\dif B_s.
\end{equation*}
Since $f$ is $C^2$ and $\sigma$ is locally bounded, the last term in the above equation is a martingale. Taking expectation on both sides of the above equation and using the fact that $\mathcal{L}f = 1$ yield
\begin{equation*}
\Enum_xf\bigl(X_{\tau_{\delta\lambda}\wedge t}\bigr) = f(x)+\Enum_x\tau_{\delta\lambda}\wedge t.
\end{equation*}
For any $\abs{x} <\delta\lambda$, letting $t\rightarrow\infty$ in the above equation yields
\begin{equation*}
\Enum_x\tau_{\delta\lambda}
=\lim_{t\rightarrow\infty} \Enum f\bigl(X_{\tau_{\delta\lambda}\wedge t}\bigr)-f(x)
\leq \sup_{\abs{y}<\delta\lambda}f(y)- f(x).
\end{equation*}
It then follows from from Chebyshev's inequality that
\begin{equation*}
\begin{split}
\sup_{\abs{x}<\delta\lambda}\Pro_{x}(X^*_{s-r} < \delta\lambda)
&= \sup_{\abs{x}<\delta\lambda}\Pro_{x}(\tau_{\delta\lambda}>s-r)
\leq \sup_{\abs{x}<\delta\lambda}\frac{\Enum_x\tau_{\delta\lambda}}{s-r} \\
&\leq \sup_{\substack{\lambda>0,\\
\abs{x}<\delta\lambda,\abs{y}<\delta\lambda}}\frac{f(y)- f(x)}{g^{-1}(\beta\lambda)-g^{-1}(\lambda)} := \phi(\delta).
\end{split}
\end{equation*}
Thus we have
\begin{equation*}
\Pro_0\bigl(g(\tau)\geq 2\lambda, X^*_\tau <\delta\lambda\bigr)
\leq \phi(\delta)\Pro_0\bigl(g(\tau)\geq \lambda\bigr).
\end{equation*}
The desired result then follows from the condition \eqref{ieq_lo} and Lemma \ref{Lemma3.3}.
\end{proof}

\begin{remark}\label{rem2}
Actually, the condition \eqref{ieq_lo} can be weakened for some specific processes. First, from the proof of Theorem \ref{lowerbound}, if $X$ is a nonnegative process, then the condition \eqref{ieq_lo} can be weakened as
\begin{equation*}
\lim_{\delta\downarrow 0}
\sup_{\substack{\lambda>0,\atop 0\leq x, y<\delta\lambda}}\frac{f(y)-f(x)}
{g^{-1}(\beta\lambda)-g^{-1}(\lambda)} = 0.
\end{equation*}
In addition, if $X$ is regular, i.e. $\Pro_x(T_y<\infty)>0$ for any $x, y\in \mathbb{R}$, where $T_y:=\inf\{t>0: X_t=y\}$ is the hitting time of $y$ \cite[Page 300]{revuz1999continuous}, then we have $\Enum_x\tau_{\delta\lambda}<\infty$ \cite[Chapter VII, Proposition 3.1]{revuz1999continuous}. From the proof of Theorem \ref{lowerbound}, the condition \eqref{ieq_lo} can be weakened as
\begin{equation*}
\lim_{\delta\downarrow 0}\sup_{\substack{\lambda>0,\\
\abs{x}<\delta\lambda}}\frac{f(\delta\lambda)\vee f(-\delta\lambda)- f(x)}{g^{-1}(\beta\lambda)-g^{-1}(\lambda)}=0.
\end{equation*}
Finally, if the function $f$ in Theorem \ref{lowerbound} is increasing on $[0,\infty)$ and decreasing on $(-\infty, 0]$, then the condition \eqref{ieq_lo} can be weakened as
\begin{equation*}
\lim_{\delta\downarrow 0}\sup_{\lambda>0}\frac{f(\delta\lambda)\vee f(-\delta\lambda)}{g^{-1}(\beta\lambda)-g^{-1}(\lambda)}=0.
\end{equation*}
\end{remark}

\section{Ornstein-Uhlenbeck processes}\label{secou}
The Ornstein-Uhlenbeck (OU) process is one of the most important kinetic models in statistical mechanics \cite{ornstein1930on, chen2020mathematical}. Let $X = (X_t)_{t\geq 0}$ be a one-dimensional OU process starting from zero, which is the unique solution to the SDE
\begin{equation}\label{OU}
\dif X_t = -\alpha X_t\dif t+\dif B_t,\;\;\;X_0 = 0,
\end{equation}
where $\alpha>0$. The moderate maximal inequalities for OU processes have been studied in \cite{jia2020moderate}. Here we revisit the moderate maximal inequalities for OU processes using the results of the present paper.

\begin{theorem}\label{themou}
Let $X = (X_t)_{t\geq 0}$ be the OU process solving \eqref{OU}. Then for any moderate function $F$, there exist two constants $c_{\alpha,F},C_{\alpha,F}>0$ such that for any stopping time of the filtration $\{\mathcal{F}_t\}$,
\begin{equation}\label{eqou}
c_{\alpha,F}\Enum F\bigl(\log^{1/2}(1+\alpha\tau)\bigr)\leq \Enum F(X^*_\tau)\leq C_{\alpha,F}\Enum F\bigl(\log^{1/2}(1+\alpha\tau)\bigr).
\end{equation}
\end{theorem}

\begin{proof}
For any $x\in R$, let $X^x$ be the solution to the SDE
\begin{equation*}
dX_t^x = -\alpha X_t^xdt+dB_t,\;\;\;X_0^x = x.
\end{equation*}
Recall that the following $L^1$ maximal inequalities for $X$ have been established \cite[Theorem 2.5]{graversen2000maximal}:
\begin{equation*}
\Enum X^*_t \leq C\Enum \log^{1/2}(1+\alpha t),\;\;\;t\geq 0.
\end{equation*}
Then it easy to check that $X_t^x = xe^{-\alpha t}+X_t^0$ for any $t\geq 0$. Hence for any $t\geq 0$ and $\lambda>0$, we have
\begin{equation*}
\sup_{\abs{x}=\lambda}\Pro_x\bigl(X^*_t\geq 2\lambda\bigr)\leq \Pro_0\bigl(X^*_t\geq \lambda\bigr),
\end{equation*}
which means that $X$ is controllable. The upper bound of \eqref{eqou} then follows for Theorem \ref{upperbound}.

On the other hand, let $f:\Rnum\rightarrow\Rnum$ and $g:\Rnum_+\rightarrow\Rnum_+$ be the functions defined as
\begin{equation*}
f(x) = 2\int_0^x e^{\alpha u^2}du\int_0^u e^{-\alpha v^2}dv,\;\;\;
g(x) = \log^{1/2}(1+\alpha x).
\end{equation*}
Then $f$ is the unique solution to the initial value problem of the ordinary differential equation (ODE)
\begin{equation*}
\mathcal{L}f = 1,\;\;\;f(0) = f'(0) = 0,
\end{equation*}
where $\mathcal{L}$ is the generator of $X$. To proceed, note that
\begin{equation*}
f(x)\leq 2\int_0^x ue^{\alpha u^2}du
= \frac{1}{\alpha}(e^{\alpha x^2}-1) = g^{-1}(\sqrt{\alpha}x).
\end{equation*}
Moreover, it is not difficult to check that $g^{-1}(ax)\geq a^2g^{-1}(x)$ for any $x\geq 0$ and $a>1$. Since $f$ is an even function, we finally obtain
\begin{equation*}
\sup_{\substack{\lambda>0}}\frac{f(\delta\lambda)\vee f(-\delta\lambda)}
{g^{-1}(2\lambda)-g^{-1}(\lambda)}
\leq \sup_{\substack{\lambda>0}} \frac{f(\delta\lambda)}{g^{-1}(\lambda)}
\leq \sup_{\substack{\lambda>0}} \frac{g^{-1}(\sqrt{\alpha}\delta\lambda)}{g^{-1}(\lambda)} \leq \alpha\delta^2,
\end{equation*}
which tends to zero as $\delta\rightarrow 0$. Note that $f$ is increasing on $[0,\infty)$ and decreasing on $(-\infty, 0]$. The lower bound of \eqref{eqou} then follows for Theorem \ref{lowerbound} and Remark \ref{rem2}.
\end{proof}

\section{Brownian motions with drift and reflected Brownian motions with drift}\label{sec5}
Next we apply our main theorems to Brownian motions with drift and reflected Brownian motions with drift. Let $V_t = B_t-\mu t$ be a Brownian motion with drift $-\mu$ starting from zero, where $\mu>0$. The definition of a reflected Brownian motion with drift is given as follows \cite{graversen2000extension}.

\begin{definition}\label{defdrift}
For any $\mu>0$, let $\beta = (\beta_t)_{t\geq 0}$ be the unique solution to the SDE
\begin{equation*}
\dif \beta_t = -\mu\;\mathrm{sign}\bigl(\beta_t\bigr)\dif t+\dif B_t,\;\;\;\beta_0 = 0.
\end{equation*}
Then $X_t = \abs{\beta_t}$ is a realization of a reflected Brownian motion with drift $-\mu$ starting from zero.
\end{definition}

Note that in \cite{peskir2001bounding}, only the upper bound of the $L^1$ maximal inequalities for $V$ is obtained. In \cite{yan2005lp}, the lower bound of the $L^p$ maximal inequalities for $V$ is obtained, but the control function in the lower bound is different from that in the upper bound and thus is not sharp. Moreover, in \cite{peskir2001bounding, yan2005lp}, the $L^p$ maximal inequalities for $X$ are only obtained for $0<p<2$. The following theorem gives the sharp moderate maximal inequalities for (reflected) Brownian motions with drift, which imply that the $L^p$ maximal inequalities hold for all $p>0$.

\begin{theorem}\label{thm4.2}
Let $V = (V_t)_{t\geq 0}$ be a Brownian motion with drift $-\mu$ starting from zero and let $X = (X_t)_{t\geq 0}$ be a reflected Brownian motion with drift $-\mu$ starting from zero, where $\mu>0$. Then for any moderate function $F$, there exist two constants $c_{\mu,F},C_{\mu,F}>0$ such that for any $\tau$ of the filtration $\{\mathcal{F}_t\}$,
\begin{align}
\label{eqref1}
c_{\mu,F}\Enum F\bigl(\log(\mu\sqrt{\tau}+1)\bigr)\leq \Enum F(V^*_\tau) \leq C_{\mu,F}\Enum F\bigl(\log(\mu\sqrt{\tau}+1)
\bigr),\\
\label{eqref2}
c_{\mu,F}\Enum F\bigl(\log(\mu\sqrt{\tau}+1)\bigr)\leq \Enum F(X^*_\tau) \leq C_{\mu,F}\Enum F\bigl(\log(\mu\sqrt{\tau}+1)\bigr).
\end{align}
\end{theorem}

\begin{proof}
We first focus on the moderate maximal inequalities for $V$. Recall that the following $L^1$ maximal inequality for $V$ has been established \cite[Equation (2.47)]{peskir2001bounding}:
\begin{equation*}
\Enum V^*_t \leq C_\mu g_\mu(t),\;\;\;t\geq 0.
\end{equation*}
To proceed, let $V_t^x = B_t+\mu t+x = V_t+x$ be a Brownian motion with drift starting from $x\in\Rnum$. For any $t\geq 0$ and $\lambda> 0$, we have
\begin{equation*}
\sup_{\abs{v}=\lambda}\Pro_v(V^*_t\geq 2\lambda)\leq \Pro_0\bigl(V^*_t\geq \lambda\bigr),
\end{equation*}
which shows that $V$ is controllable. Let $f_{\mu}:\Rnum\rightarrow\Rnum_+$ be the function defined as
\begin{equation}\label{eq4.1}
f_{\mu}(x) = \frac{e^{2\mu x}-2\mu x -1}{2\mu^2},
\end{equation}
and let $g_{\mu}:\Rnum_+\rightarrow\Rnum_+$ denote the inverse of $f_{\mu}$ for $x\geq 0$. It is easy to check that for any $x\geq 0$,
\begin{equation*}
\frac{1}{\mu^2}\left(e^{\mu x}-2e^{\mu x/2}+1\right)\leq f_\mu(x)\leq  \frac{1}{\mu^2}\left(e^{4\mu x}-2e^{\mu x}+1\right).
\end{equation*}
This indicates that
\begin{equation}\label{eqgmu}
\frac{1}{2\mu}\log(\mu\sqrt{x}+1)\leq g_\mu(x)\leq  \frac{2}{\mu}\log(\mu\sqrt{x}+1).
\end{equation}
The upper bound then follows from Theorem \ref{upperbound} and the inequality \eqref{eqgmu}. On the other hand, it is easy to check $f_\mu$ is the unique solution to the initial value problem of the ODE
\begin{equation*}
\mathcal{L}f = 1,\;\;\;f(0) = f'(0) = 0,
\end{equation*}
where $\mathcal{L}$ is the generator of $V$, and has the following integral representation:
\begin{equation*}
f_\mu(x) = \int_0^xh_{\mu}(u)du,
\end{equation*}
where $h_{\mu}(u) = 2e^{2\mu u}\int_{0}^{u}e^{-2\mu v}dv$ is an strictly increasing function for $u\geq 0$. Hence for any $\lambda>0$, we have
\begin{equation*}
f_{\mu}(2\lambda)-f_{\mu}(\lambda) = \int_\lambda^{2\lambda}h_{\mu}(u)du
> \int_0^\lambda h_{\mu}(u)du = f_{\mu}(\lambda)
\end{equation*}
and for any $0<\delta<1$,
\begin{equation*}
f_{\mu}(\delta\lambda) = \int_0^{\delta\lambda}h_{\mu}(u)du
= \delta\int_0^\lambda h_{\mu}(\delta u)du < \delta f_{\mu}(\lambda).
\end{equation*}
Moreover, it is easy to check that $f_{\mu}(x)>f_{\mu}(-x)$ for any $x>0$. Thus we finally obtain
\begin{equation}\label{drifted low}
\sup_{\lambda>0}\frac{f_\mu(\delta\lambda)\vee f_\mu(-\delta\lambda)}{f_\mu(2\lambda)-f_\mu(\lambda)}
\leq \sup_{\lambda> 0}\frac{f_\mu(\delta\lambda)}{f_\mu(\lambda)} \leq \delta,
\end{equation}
which tends to zero as $\delta\rightarrow 0$. Note that $f_\mu$ is increasing on $[0,\infty)$ and decreasing on $(-\infty, 0]$. The lower bound then follows from Theorem \ref{lowerbound}, Remark \ref{rem2}, and the inequality \eqref{eqgmu}.

We next focus on the moderate maximal inequalities for $X$. Recall that the following $L^1$ maximal inequality for $X$ has been established \cite[Theorem 2.1]{peskir1999maximal}:
\begin{equation*}
\Enum X^*_t \leq C_\mu g_\mu(t),\;\;\;t\geq 0.
\end{equation*}
Next we will prove that for any $t\geq 0$ and $\lambda>0$,
\begin{equation}\label{temp3}
\Pro_{\lambda}(\beta^*_t\geq 2\lambda)\leq 2\Pro_0\bigl(\beta^*_t\geq \lambda\bigr).
\end{equation}
For any $x>0$, let $\tau_x = \inf\{t\geq 0:|\beta_t|\geq x\}$ and $\eta = \inf\{t\geq 0:\beta_t=0\}$. It is easy to see that
\begin{align*}
\Pro_{\lambda}(\beta^*_t\geq 2\lambda)
= \Pro_{\lambda}(\tau_{2\lambda}\leq t)
= \Pro_{\lambda}(\tau_{2\lambda}\leq t,\tau_{2\lambda}\geq \eta)
+\Pro_{\lambda}(\tau_{2\lambda}\leq t,\tau_{2\lambda}<\eta) := \mathrm{I}+\mathrm{II}.
\end{align*}
By the strong Markov property of $\beta$, we have
\begin{equation}\label{temp1}
\begin{split}
\mathrm{I} =&\; \Pro_{\lambda}(\tau_{2\lambda}\leq t,\tau_{2\lambda}\geq\eta)\\
=&\; \Enum_{\lambda}\bigl[I_{\{\tau_{2\lambda}\geq\eta\}}
\Enum_{\lambda}(I_{\{\eta\leq\tau_{2\lambda}\leq t\}}|\mathcal{F}_{\eta})\bigl]\\
\leq&\; \Enum_{\lambda}\bigl[I_{\{\tau_{2\lambda}\geq \eta\}}
\Pro_{0}(\tau_{2\lambda}\leq t)\bigl]
\leq \Pro_0(\tau_{2\lambda}\leq t)
\leq \Pro_0\bigl(\beta^*_t\geq\lambda\bigr).
\end{split}
\end{equation}
To proceed, let $Y^\lambda_t = \beta_t+\lambda$ for any $t\geq 0$. Then $Y^\lambda = (Y^\lambda_t)_{t\geq 0}$ is the solution to the SDE
\begin{equation*}
dY^\lambda_t = -\mu\;\mathrm{sign}\bigl(Y^\lambda_t-\lambda)\dif t+\dif B_t,\;\;\;Y^\lambda_0 = \lambda.
\end{equation*}
Moreover, let $\beta^\lambda = (\beta^\lambda_t)_{t\geq 0}$ be the solution to the SDE
\begin{equation*}
d\beta^\lambda_t = -\mu\;\mathrm{sign}\bigl(\beta^\lambda_t)\dif t+\dif B_t,\;\;\;\beta^\lambda_0 = \lambda.
\end{equation*}
It then follows from the comparison theorem that $\beta^\lambda_t\leq Y^\lambda_t$ for any $t\geq 0$ (with probability one). Note that here we did not use the classical version of the comparison theorem, which requires that at least one of the drift terms of $\beta^\lambda$ and $Y^\lambda$ satisfies the Lipschitz condition \cite[Chapter IX, Theorem 3.7]{revuz1999continuous}. Rather here we use the version stated in \cite[Chapter VI, Theorem 1.1]{Ikeda}. Specifically, let $b_1,b_2:\Rnum\rightarrow\Rnum$ be two continuous functions defined as
\begin{equation*}
b_1(x) = \begin{cases}
\mu,& x\leq 0,\\
\mu(1-4x/\lambda), &0<x<\lambda/2,\\
-\mu, &x\geq\lambda/2,
\end{cases}
\end{equation*}
and $b_2(x)=b_1(x-\lambda/2)$. For any $x\in\Rnum$, it is easy to see that
\begin{equation*}
b_1(x)\leq b_2(x),\;\;\;
-\mu\;\mathrm{sign}(x)\leq b_1(x),\;\;\;
-\mu\;\mathrm{sign}(x)\geq b_2(x).
\end{equation*}
Moreover, since the drift and diffusion terms of both $\beta^\lambda$ and $Y^\lambda$ are bounded and the diffusion terms of both processes are constants, it follows from the Nakao-Le Gall uniqueness theorem \cite[Chapter V, Theorem 41.1]{rogers2000diffusions2} that both processes are pathwise unique. Then the comparison theorem \cite[Chapter VI, Theorem 1.1]{Ikeda} indicates that $\beta^\lambda_t\leq Y^\lambda_t$ for any $t\geq 0$. For any $x>0$, let $\gamma_x  = \inf\{t\geq 0:Y^\lambda_t\geq x\}$ and let $\sigma_x  = \inf\{t\geq 0:\beta^\lambda_t\geq x\}$. Then we have
\begin{equation}\label{temp2}
\mathrm{II} = \Pro_{\lambda}(\tau_{2\lambda}\leq t,\tau_{2\lambda}<\eta)
= \Pro(\sigma_{2\lambda}\leq t) \leq \Pro(\gamma_{2\lambda}\leq t)
\leq \Pro_0\bigl(\beta^*_t\geq \lambda\bigr).
\end{equation}
Combining \eqref{temp1} and \eqref{temp2}, we obtain \eqref{temp3}. Similarly, we can prove that
\begin{equation*}
\Pro_{-\lambda}(\beta^*_t\geq 2\lambda)\leq 2\Pro_0\bigl(\beta^*_t\geq \lambda\bigr).
\end{equation*}
This equation, together with \eqref{temp3}, shows that $\beta$ is controllable. The upper bound of \eqref{eqref2} then follows from Theorem \ref{upperbound} and \eqref{eqgmu}. On the other hand, let $\tilde{f}:\Rnum\rightarrow\Rnum$ be the function defined as
\begin{equation*}
\tilde{f}_{\mu}(x) = \frac{e^{2\mu \abs{x}}-2\mu \abs{x} -1}{2\mu^2}.
\end{equation*}
It is straightforward to check that $\tilde{f}_{\mu}\in C^2(\Rnum)$ and
\begin{equation*}
\mathcal{L_{\mu}}\tilde{f}_{\mu} = 1,\;\;\;\tilde{f}_{\mu}(0) = \tilde{f}_{\mu}'(0) = 0,
\end{equation*}
where $\mathcal{L_{\mu}}$ is the generator of $\beta$. Since $\tilde{f}_{\mu}(x) = f_{\mu}(|x|)$ for any $x\in\Rnum$, it follows from \eqref{drifted low} that
\begin{equation*}
\lim_{\delta\downarrow 0}\sup_{\lambda>0}\frac{\tilde{f}_\mu(\delta\lambda)\vee \tilde{f}_\mu(-\delta\lambda)}{f_\mu(2\lambda)-f_\mu(\lambda)} = 0.
\end{equation*}
Applying Theorem \ref{lowerbound} and Remark \ref{rem2} to $\beta$ and noting that the maximal processes of $\beta$ and $X$ are the same give the lower bound of \eqref{eqref2}.
\end{proof}

\section{Cox-Ingersoll-Ross processes, radial Ornstein-Uhlenbeck processes, and Bessel processes}
In mathematical finance, the Cox-Ingersoll-Ross (CIR) model is widely applied to describe the evolution of interest rates \cite{cox1985theory}. We first recall the following definition \cite{yan2004maximal, going2003survey}.

\begin{definition}\label{defCIR}
For any $a\geq 0$, $b\in\mathbb{R}$, and $c>0$, the unique solution $C = (C_t)_{t\geq 0}$ to the SDE
\begin{equation}\label{CIR}
\dif C_t = \bigl(a+bC_t\bigr)\dif t+c\sqrt{\abs{C_t}}\dif B_t,\;\;\;C_0 = x\geq 0,
\end{equation}
is called the CIR process starting from $x$ and is denoted by $\mathrm{CIR}(a,b,c,x)$.
\end{definition}

The comparison theorem ensures that $C_t\geq 0$ for any $t\geq 0$ \cite{going2003survey}. When $a = 0$ and $x = 0$, the solution to \eqref{CIR} is the constant process $C = 0$. In this case, the maximal inequalities for $C$ is trivial and thus we assume $a>0$ in the following. In the special case of $a = \alpha>0$, $b=0$, and $c=2$, the CIR process reduces to a squared Bessel process, which is defined below \cite[Chapter XI, Definition 1.1]{revuz1999continuous}.

\begin{definition}
For any $\alpha>0$, the unique solution to the SDE
\begin{equation*}
\dif Y^\alpha_t = \alpha\dif t+2\sqrt{Y^\alpha_t}\dif B_t,\;\;\;Y^\alpha_0 = x\geq 0,
\end{equation*}
is called a squared Bessel process starting from $x$ and is denoted by $\mathrm{BESQ}(\alpha,x)$.
\end{definition}

We first study the moderate maximal inequalities for squared Bessel processes.

\begin{theorem}\label{BESQ}
Let $Y^\alpha = (Y^\alpha_t)_{t\geq 0}$ be a $\mathrm{BESQ}(\alpha,0)$. Then for any moderate function $F$, there exist two constants $c_{\alpha,F},C_{\alpha,F}>0$ such that for any stopping time $\tau$ wth respect to the filtration $\{\mathcal{F}_t\}$,
\begin{equation}\label{eqBESQ}
c_{\alpha,F}\Enum F(\tau)\leq \Enum F(Y^{\alpha,*}_\tau)\leq C_{\alpha,F}\Enum F(\tau).
\end{equation}
\end{theorem}

\begin{proof}
We first prove that there exist constants $C,\gamma>0$, such that for any $t\geq 0$ and $\lambda>0$,
\begin{equation}\label{equcon1}
\Pro_{\lambda}(Y^{\alpha,*}_t\geq 4\lambda)\leq C\Pro_0\bigl(Y^{\alpha,*}_t\geq \gamma\lambda\bigr).
\end{equation}
To the end, we consider the following three cases.

The first case occurs when $\alpha = N$ is a positive integer. Recall that a squared Bessel process of dimension $N$ is the same in law as the square of the Euclidean norm of an $N$-dimensional Brownian motion \cite[Page 439]{revuz1999continuous}. Specifically, let $W_t = (W_{1,t},W_{2,t},\ldots, W_{N,t})$ be a $N$-dimensional standard Brownian motion. Then the process $Y^N_t := \sum_{i=1}^{N}W_{i,t}^2$ with $\sum_{i=1}^{N}W_{i,0}^2 = \lambda$ (e.g., with $W_{1,0} = \sqrt{\lambda}$ and $W_{2,0} = \cdots = W_{N,0} = 0$) is a $\mathrm{BESQ}(N,\lambda)$. Hence
\begin{align*}
\Pro_{\lambda}\left(Y^{N,*}_t\geq 4\lambda\right)
=&\; \Pro_{(\sqrt{\lambda},0,\cdots,0)}\bigg(\sup_{0\leq s\leq t}\sum_{i=1}^{N}W_{i,s}^2\geq 4\lambda\bigg)\\
=&\; \Pro_{(\sqrt{\lambda},0,\cdots,0)}
\bigg(\sup_{0\leq s\leq t}\sum_{i=1}^{N}\left(W_{i,s}-W_{i,0}+W_{i,0}\right)^2\geq 4\lambda\bigg)\\
\leq&\; \Pro_{(\sqrt{\lambda},0,\cdots,0)}\bigg(\sup_{0\leq s\leq t}\sum_{i=1}^{N}\left(W_{i,s}-W_{i,0}\right)^2\geq \lambda\bigg)
=\Pro_{0}\left(Y^{N,*}_t\geq \lambda\right).
\end{align*}

The second case occurs when $\alpha>1$ and $\alpha$ is not an integer. Let $\lceil\alpha\rceil$ denote the smallest integer larger than $\alpha$. By the comparison theorem and case 1, we have
\begin{align*}
\Pro_{\lambda}\left(Y^{\alpha,*}_t\geq 4\lambda\right)
\leq&\; \Pro_{\lambda}\left(Y^{\lceil\alpha\rceil,*}_t\geq 4\lambda\right)
\leq\Pro_0\left(Y^{\lceil\alpha\rceil,*}_t\geq \lambda\right)\\
=&\; \Pro_{(0,\cdots,0)}\bigg(\sup_{0\leq s\leq t}\sum_{i=1}^{\lceil\alpha\rceil}W_{i,s}^2\geq \lambda\bigg)\\
\leq&\; \Pro_{(0,\cdots,0)}\bigg(\sup_{0\leq s\leq t}\sum_{i=1}^{\lceil\alpha\rceil-1}W_{i,s}^2\geq \frac{1}{2}\lambda\;\;\mathrm{or}\;\sup_{0\leq s\leq t}W_{\lceil\alpha\rceil,s}^2\geq \frac{1}{2}\lambda\bigg)\\
\leq&\; 2\Pro_{(0,\cdots,0)}\bigg(\sup_{0\leq s\leq t}\sum_{i=1}^{\lceil\alpha\rceil-1}W_{i,s}^2\geq \frac{1}{2}\lambda\bigg)\\
=&\; 2\Pro_0\left(Y^{\lceil\alpha\rceil-1,*}_t\geq \frac{1}{2}\lambda\right)\leq 2\Pro_0\left(Y^{\alpha,*}_t\geq \frac{1}{2}\lambda\right).
\end{align*}

The third case occurs when $0<\alpha<1$. We will next prove by induction that for any $k\geq 1$,
\begin{equation}\label{equcon2}
\Pro_{\lambda}\left(Y^{\frac{1}{2^k},*}_t\geq 4\lambda\right)
\leq 2^k\Pro_0\left(Y^{\frac{1}{2^k},*}_t\geq \frac{1}{2^k}\lambda\right).
\end{equation}
To this end, we recall the following additive property for squared Bessel Processes \cite[Chapter XI, Theorem 1.2]{revuz1999continuous}: for any $x, x'\geq 0$ and $\alpha,\alpha'>0$, if $Y^{\alpha}\sim\mathrm{BESQ}(\alpha,x)$ and $Y^{\alpha'}\sim\mathrm{BESQ}(\alpha',x')$ are independent, then we have $Y^{\alpha}+Y^{\alpha'}\sim\mathrm{BESQ}(\alpha+\alpha',x+x')$. Hence by the comparison theorem and case 1, we obtain
\begin{align*}
\Pro_{\lambda}\left(Y^{\frac{1}{2},*}_t\geq 4\lambda\right)
\leq&\; \Pro_{\lambda}\left(Y^{1,*}_t\geq 4\lambda\right)
\leq\Pro_0\left(Y^{1,*}_t\geq \lambda\right)\\
=&\; \Pro_0\left(\big(Y^{\frac{1}{2}}+\tilde{Y}^{\frac{1}{2}}\big)^*_t\geq \lambda\right)\\
\leq&\; 2\Pro_0\left(Y^{\frac{1}{2},*}_t\geq \frac{1}{2}\lambda\right),
\end{align*}
where $\tilde{Y}^{\frac{1}{2}}$ is an independent copy of $Y^{\frac{1}{2}}$.
Suppose that (\ref{equcon2}) holds for some $k\geq 1$. Then
\begin{align*}
\Pro_{\lambda}\left(Y^{\frac{1}{2^{k+1}},*}_t\geq 4\lambda\right)
\leq&\; \Pro_{\lambda}\left(Y^{\frac{1}{2^k},*}_t\geq 4\lambda\right)
\leq 2^k\Pro_0\left(Y^{\frac{1}{2^k},*}_t\geq \frac{1}{2^k}\lambda\right)\\
=&\; 2^k\Pro_0\left(\left(Y^{\frac{1}{2^{k+1}}}+\tilde{Y}^{\frac{1}{2^{k+1}}}\right)^*_t\geq \frac{1}{2^{k}}\lambda\right)\\
\leq&\; 2^{k+1}\Pro_0\left(Y^{\frac{1}{2^{k+1}},*}_t\geq \frac{1}{2^{k+1}}\lambda\right),
\end{align*}
where $\tilde{Y}^{\frac{1}{2^{k+1}}}$ is an independent copy of $Y^{\frac{1}{2^{k+1}}}$. Hence by induction, \eqref{equcon2} holds for any $k\geq 1$. Finally, for any $\alpha>0$, there exists $k\geq 1$ such that $\alpha\in(1/2^{k+1},1/2^k]$. It then follows from the comparison theorem and \eqref{equcon2} that
\begin{align*}
\Pro_{\lambda}\left(Y^{\alpha,*}_t\geq 4\lambda\right)
\leq&\; \Pro_{\lambda}\left(Y^{\frac{1}{2^k},*}_t\geq 4\lambda\right)
\leq 2^n\Pro_0\left(Y^{\frac{1}{2^k},*}_t\geq \frac{1}{2^k}\lambda\right)\\
\leq&\; \frac{1}{2^{k+1}}\Pro_0\left(Y^{\frac{1}{2^{k+1}},*}_t\geq \frac{1}{2^{k+1}}\lambda\right)
\leq \frac{1}{2^{k+1}}\Pro_0\left(Y^{\alpha,*}_t\geq\frac{1}{2^{k+1}}\lambda\right).
\end{align*}

To summarize, for any $\alpha>0$, we have proved that
\begin{equation}\label{equcon3}
\Pro_{\lambda}\left(Y^{\alpha,*}_t\geq 4\lambda\right)
\leq 2^{\lceil \frac{1}{\alpha}\rceil}\Pro_0\left(Y^{\alpha,*}_t\geq 2^{-\lceil\frac{1}{\alpha}\rceil}\lambda\right),
\end{equation}
which shows that $Y^\alpha$ is controllable. Moreover, recall that for any $0<p<1$, the following $L^p$ maximal inequality for $Y^\alpha$ has been established \cite[Equation (3.16)]{yan2004maximal}
\begin{equation} \label{CIRL1}
\Enum (Y^{\alpha,*}_t)^p \leq \alpha^p\frac{2-p}{1-p}t^p,\;\;\;t\geq 0.
\end{equation}
The upper bound of \eqref{eqBESQ} then follows from Theorem \ref{upperbound}.

We next prove the lower bound. Let $f_{\alpha}(x) = x/\alpha$ for any $x\in\Rnum$. It is easy to check that
\begin{equation*}
\mathcal{L}f_{\alpha} = 1,\;\;\;f_{\alpha}(0) = 0,\;\;\;f_{\alpha}'(0) = \frac{1}{\alpha}.
\end{equation*}
where $\mathcal{L}$ is the generator of $Y^\alpha$. Note that
\begin{equation*}
\lim_{\delta\downarrow 0}\sup_{\lambda>0}\frac{f_{\alpha}(\delta\lambda)}{2\lambda-\lambda} = 0.
\end{equation*}
The lower bound then follows from Theorem \ref{lowerbound} and Remark \ref{rem2}.
\end{proof}

\begin{remark}
In \cite[Theorem 3.1]{yan2004maximal}, the authors have established the $L^p$ maximal inequalities for $Y^\alpha$ for any $p>0$. However, their proof when $p\geq 1$ is questionable, because they mistakenly regarded the random time $TI_{\{T>S\}}$ as a stopping time, when $S$ and $T$ are two stopping times with $S\leq T$ (see the last paragraph in page 119 of \cite{yan2004maximal}).
\end{remark}

We then apply the above theorem to Bessel processes \cite[Chapter XI, Definition 1.9]{revuz1999continuous}.

\begin{definition}\label{defbessel}
For any $\alpha>0$ and $x\geq 0$, the square root of the process $\mathrm{BESQ}(\alpha,x^2)$ is called a Bessel process of dimension $\alpha$ starting from $x$ and is denoted by $\mathrm{BES}(\alpha,x)$.
\end{definition}

Bessel processes may or may not be diffusions \cite{going2003survey}. A Bessel process of dimension $\alpha>1$ starting from $x>0$ is a submartingale and the solution to the SDE
\begin{equation}\label{eq5.1}
\dif{U^\alpha_t} = \frac{\alpha -1}{2U^\alpha_t}\dif t+\dif B_t,\;\;\; U^\alpha_0 = x.
\end{equation}
A Bessel process of dimension $\alpha = 1$ can be realized by a reflected Brownian motion, which is a submartingale but is not a diffusion in the sense of \eqref{model}. A Bessel process of dimension $0<\alpha<1$ is not even a semimartingale and thus is not a diffusion. Please refer to \cite{graversen1998maximal, going2003survey} for details.

Note that in \cite{dubins1994optimal, graversen1998maximal}, the $L^p$ maximal inequalities for $U^\alpha$ are obtained for $\alpha\geq 1$ and $p>0$. Moreover, in \cite{yan2005lp}, the $L^p$ maximal inequalities for $U^\alpha$ are obtained for $\alpha>0$ and $0<p<2$. The following theorem gives the sharp moderate maximal inequalities for Bessel processes, which imply that the $L^p$ maximal inequalities hold for all $\alpha>0$ and $p>0$.

\begin{corollary}\label{bessel}
Let $U^\alpha = (U^\alpha_t)_{t\geq 0}$ be a $\mathrm{BES}(\alpha,0)$. Then for any moderate function $F$, there exist two constants $c_{\alpha,F},C_{\alpha,F}>0$ such that for stopping time $\tau$ of the filtration $\{\mathcal{F}_t\}$,
\begin{equation}\label{eqBES}
c_{\alpha,F}\Enum F(\sqrt{\tau})\leq \Enum F(U^{\alpha,*}_\tau)\leq C_{\alpha,F}\Enum F(\sqrt{\tau}).
\end{equation}
\end{corollary}

\begin{proof}
Let $f(x) := F(\sqrt{x})$ for any $x\geq 0$. Since $F$ is a moderate function, it is easy to check that $f$ is also a moderate function. The desired result then follows from Theorem \ref{BESQ}.
\end{proof}

The following theorem gives the moderate maximal inequalities for CIR process with $b<0$.

\begin{theorem}\label{thmCIR}
Let $C = (C_t)_{t\geq 0}$ be a $\mathrm{CIR}(a,b,c,0)$ with $a,c>0$ and $b<0$. Then for any moderate function $F$, there exist two constants $c_{F},C_{F}>0$ depending on $a$, $b$, and $c$ such that for any stopping time $\tau$ wth respect to the filtration $\{\mathcal{F}_t\}$,
\begin{equation*}
c_{F}\Enum F\left(\log\left(1-\frac{2ab}{c^2}\tau\right)\right)
\leq \Enum F(C_\tau^*) \leq C_{F}\Enum F\left(\log\left(1-\frac{2ab}{c^2}\tau\right)\right).
\end{equation*}
\end{theorem}

\begin{proof}
We first consider the upper bound. From \cite[Equation (4)]{going2003survey}, any CIR process $C\sim\mathrm{CIR}(a,b,c,x)$ can be represented by
\begin{equation}\label{rep}
C_t = e^{bt}Y^{\alpha}_{\frac{c^2}{4b}(1-e^{-bt})},
\end{equation}
where $Y^\alpha$ is a $\mathrm{BESQ}(\alpha,x)$ with $\alpha = 4a/c^2$. We first prove that there exist constants $C,\gamma>0$, such that for any $t\geq 0$ and $\lambda>0$,
\begin{equation}\label{equCIR}
\Pro_{\lambda}(C^{*}_t\geq 4\lambda)\leq C\Pro_0\bigl(C^{*}_t\geq \gamma\lambda\bigr).
\end{equation}
The proof of \eqref{equCIR} is similar to that of \eqref{equcon1} with some modifications. To prove the above inequality, we consider the following three cases.

The first case occurs when $\alpha = N$ is a positive integer. Let $\rho(t) = \frac{c^2}{4b}(1-e^{-bt})$ for any $t\geq 0$ and let $W_t=(W_{1,t},W_{2,t},\ldots, W_{N,t})$ be a $N$-dimensional standard Brownian motion. Since $b<0$, we have
\begin{align*}
\Pro_{\lambda}(C^{*}_t\geq 4\lambda)
=&\; \Pro_\lambda\left(\sup_{0\leq s\leq t}\left(e^{bs}Y^{\alpha}_ {\rho(s)}\right)\geq 4\lambda\right)\\
=&\; \Pro_{(\sqrt{\lambda},0,\cdots,0)}\left(\sup_{0\leq s\leq t}\left(e^{bs}\sum_{i=1}^{N}W_{i,\rho(s)}^2\right)\geq 4\lambda\right)\\
=&\; \Pro_{(\sqrt{\lambda},0,\cdots,0)}\left(\sup_{0\leq s\leq t}\left(e^{bs}\sum_{i=1}^{N}\left(W_{i,\rho(s)}-W_{i,0}+W_{i,0}\right)^2\right)\geq 4\lambda\right)\\
\leq&\; \Pro_{(\sqrt{\lambda},0,\cdots,0)}\left(\sup_{0\leq s\leq t}\left(e^{bs}\sum_{i=1}^{N}\left(W_{i,\rho(s)}-W_{i,0}\right)^2\right)\geq \lambda\right)
=\Pro_{0}\left(C^{*}_t\geq \lambda\right).
\end{align*}

The second case occurs when $\alpha>1$ and $\alpha$ is not an integer. Let $\lceil\alpha\rceil$ denote the smallest integer larger than $\alpha$. By the comparison theorem and case 1, we have
\begin{align*}
\Pro_{\lambda}\left(C^{*}_t\geq 4\lambda\right)
\leq&\; \Pro_{\lambda}\left(\sup_{0\leq s\leq t}\left(e^{bs}Y^{\lceil\alpha\rceil}_{\rho(s)}\right)\geq 4\lambda\right)
\leq\Pro_0\left(\sup_{0\leq s\leq t}\left(e^{bs}Y^{\lceil\alpha\rceil}_{\rho(s)}\right)\geq \lambda\right)\\
=&\; \Pro_{(0,\cdots,0)}\left(\sup_{0\leq s\leq t}\left(e^{bs}\sum_{i=1}^{\lceil\alpha\rceil}W_{i,\rho(s)}^2\right)\geq \lambda\right)\\
\leq&\; \Pro_{(0,\cdots,0)}\left(\sup_{0\leq s\leq t}\left(e^{bs}\sum_{i=1}^{\lceil\alpha\rceil-1}W_{i,\rho(s)}^2\right)\geq \frac{1}{2}\lambda\;\;\mathrm{or}\;\sup_{0\leq s\leq t}\left(e^{bs}W_{\lceil\alpha\rceil,\rho(s)}^2\right)\geq \frac{1}{2}\lambda\right)\\
\leq&\; 2\Pro_{(0,\cdots,0)}\left(\sup_{0\leq s\leq t}\left(e^{bs}\sum_{i=1}^{\lceil\alpha\rceil-1}W_{i,\rho(s)}^2\right)\geq \frac{1}{2}\lambda\right)\\
=&\; 2\Pro_0\left(\sup_{0\leq s\leq t}\left(e^{bs}Y^{\lceil\alpha\rceil-1}_{\rho(s)}\right)\geq \frac{1}{2}\lambda\right)\\
\leq&\; 2\Pro_0\left(\sup_{0\leq s\leq t}\left(e^{bs}Y^{\alpha}_{\rho(s)}\right)\geq \frac{1}{2}\lambda\right)
= 2\Pro_0\left(C^{*}_t\geq \frac{1}{2}\lambda\right).
\end{align*}

The third case occurs when $0<\alpha<1$. From the additive property for squared Bessel Processes and the representation \eqref{rep}, it is easy to obtain the following additive property for CIR processes: for any $a,a'>0$ and $x,x'\geq 0$, if $C\sim\mathrm{CIR}(a,b,c,x)$ and $C'\sim\mathrm{CIR}(a',b,c,x')$ are independent, then we have $C+C^{'}\sim\mathrm{CIR}(a+a',b,c,x+x')$. Then the proof of \eqref{equCIR} in the third case is the same as the proof of \eqref{equcon1} in the third case.

To summarize, we have proved \eqref{equCIR}, which shows that $C$ is controllable. We shall next establish the $L^p$ maximal inequality of $C$ for $0<p<1$ following the classical method of the Lenglart domination principle \cite[Lemma 2.1]{peskir2001bounding}. Let $f:\Rnum_+\rightarrow\Rnum_+$ be the function defined as
\begin{equation*}
f(x) = \frac{2}{c^2}\int_0^x t^{-\frac{2a}{c^2}}e^{-\frac{2b}{c^2}t}\dif t\int_0^t s^{\frac{2a}{c^2}-1}e^{\frac{2b}{c^2}s}\dif s.
\end{equation*}
Complex but straightforward calculations show that $f\in C^2(\Rnum_+)$ and
\begin{equation*}
\mathcal{L}f = 1,\;\;\;f(0)=0,\;\;\;f'(0)=\frac{1}{a},
\end{equation*}
where $\mathcal{L}$ is the generator of $C$. It then follows from It\^{o}'s formula that
\begin{equation}\label{eqcirtau}
\Enum f(X_{\tau})=\Enum\tau
\end{equation}
for any bounding stopping time $\tau$ of the filtration $\{\mathcal{F}_t\}$. For any $x\geq 0$, we have
\begin{equation}\label{eqf1ff2}
\begin{split}
f_1(x):=\frac{c^2}{-ab2^{2a/c^2}}\left(e^{-\frac{b}{c^2}x}-1\right)
=&\; \frac{2}{c^2}\int_0^x t^{-\frac{2a}{c^2}}e^{-\frac{2b}{c^2}t}
\left(e^{\frac{b}{c^2}t}\int_0^{t/2} s^{\frac{2a}{c^2}-1}\dif s\right)\dif t\\
\leq&\; \frac{2}{c^2}\int_0^x t^{-\frac{2a}{c^2}}e^{-\frac{2b}{c^2}t}\dif t
\int_0^{t/2} s^{\frac{2a}{c^2}-1}e^{\frac{2b}{c^2}s}\dif s\\
\leq&\; f(x)\leq
\frac{2}{c^2}\int_0^x t^{-\frac{2a}{c^2}}e^{-\frac{2b}{c^2}t}\dif t\int_0^t s^{\frac{2a}{c^2}-1}\dif s\\
=&\; -\frac{c^2}{2ab}\left(e^{-\frac{2b}{c^2}x}-1\right):=f_2(x).
\end{split}
\end{equation}
Let $g:\Rnum_+\rightarrow\Rnum_+$ be the function defined as
\begin{equation*}
g(x)= -\frac{c^2}{2b}\log\left(1-\frac{2ab}{c^2}x\right).
\end{equation*}
Since $f$, $f_1$, and $f_2$ are both strictly increasing and vanishes at zero, we obtain
\begin{equation}\label{eqf1ff22}
g(x) = f_2^{-1}(x) \leq f^{-1}(x) \leq f_1^{-1}(x) = -\frac{c^2}{b}\log\left(1-\frac{ab2^{2a/c^2}}{c^2}x\right).
\end{equation}
For any $0<p<1$, let $H_p(x)=(f^{-1}_1(x))^p$ and
\begin{equation*}
\tilde{H}_p(x) = x\int_x^{\infty}\frac{1}{s}dH_p(x)+2H_p(x)
\end{equation*}
for any $x\geq 0$. It is easy to check that
\begin{equation*}
\lim_{x\rightarrow 0}\frac{x}{H_p(x)}\int_x^{\infty}\frac{1}{s}dH_p(x) = \frac{p}{1-p},\;\;\; \lim_{x\rightarrow\infty}\frac{x}{H_p(x)}\int_x^{\infty}\frac{1}{s}dH_p(x) = 0.
\end{equation*}
Hence we obtain
\begin{equation}\label{eqHp}
\sup_{x\geq 0}\frac{\tilde{H}_p(x)}{H_p(x)} < \infty.
\end{equation}
By \eqref{eqcirtau}, \eqref{eqHp}, and the Lenglart domination principle \cite[Lemma 2.1]{peskir2001bounding}, we obtain the following $L^p$ maximal inequality for any $0<p<1$:
\begin{equation*}
\Enum\sup_{0\leq t\leq \tau}X_t^p = \Enum\sup_{0\leq t\leq \tau}(f^{-1}(f(X_t)))^p
\leq \Enum\sup_{0\leq t\leq \tau}H_p(f(X_t)
\leq \Enum\tilde{H}_p(\tau)\lesssim\Enum H_p(\tau)\lesssim\Enum g^p(\tau),
\end{equation*}
where $x\lesssim y$ means that there exists a constant $C>0$ depending only on $p$ such that $x\leq Cy$. The upper bound of \eqref{eqBES} then follows from Theorem \ref{upperbound}.

We next consider the lower bound. By \eqref{eqf1ff2} and \eqref{eqf1ff22}, we have
\begin{align*}\label{ieq_lo1}
\lim_{\delta\downarrow 0}
\sup_{\lambda>0}\frac{f(\delta\lambda)}{g^{-1}(2\lambda)-g^{-1}(\lambda)}
&\leq \lim_{\delta\downarrow 0}
\sup_{\lambda>0}\frac{ -\frac{c^2}{2ab}(e^{-\frac{2b}{c^2}\delta\lambda}-1) }
{-\frac{c^2}{2ab}(e^{-\frac{4b}{c^2}\lambda}-1)+\frac{c^2}{2ab}(e^{-\frac{2b}{c^2}\lambda}-1)}\\
&= \lim_{\delta\downarrow 0}
\sup_{\lambda>0}\frac{ \int_0^{\delta \lambda}e^{-\frac{2b}{c^2}t}dt }
{\int_{\lambda}^{2\lambda}e^{-\frac{2b}{c^2}t}dt}
 \leq\lim_{\delta\downarrow 0} \delta=0.
\end{align*}
Note that $f$ is increasing on $[0,\infty)$. The lower bound of \eqref{eqBES} then follows from Theorem \ref{lowerbound} and Remark \ref{rem2}.
\end{proof}

\begin{remark}
In \cite[Theorem 2.1]{yan2004maximal}, the authors have established the $L^p$ maximal inequalities for $C$ for any $p>0$. However, their proof when $p\geq 1$ is questionable, because they mistakenly regarded the random time $TI_{\{T>S\}}$ as a stopping time, when $S$ and $T$ are two stopping times with $S\leq T$ (see the last paragraph in page 117 in \cite{yan2004maximal}).
\end{remark}

We then apply the above theorem to radial OU processes \cite{going2003survey}.

\begin{definition}\label{defradial}
For any $\alpha\geq 0$, $\beta\in\mathbb{R}$, and $x\geq 0$, the square root of the process $\mathrm{CIR}(\alpha,2\beta,2,x^2)$ is called a radial Ornstein-Uhlenbeck process of dimension $\alpha$ and parameter $\beta$ starting from $x$, and is denoted by $\mathrm{ROU}(\alpha,\beta,x)$.
\end{definition}

It is known that for any $\alpha>1$, $\beta>0$, and $x>0$, an $\mathrm{ROU}(\alpha,\beta,x)$ is the solution to the SDE \cite{botnikov2006davis}
\begin{equation*}
dR_t = \left(\frac{\alpha-1}{2R_t}-\beta R_t\right)\dif t+\dif B_t,\;\;\;R_0 = x.
\end{equation*}

\begin{corollary}\label{cororadial}
Let $R = (R_t)_{t\geq 0}$ be an $\mathrm{ROU}(\alpha,\beta,0)$ with $\alpha,\beta>0$. Then for any moderate function $F$, there exist two constants $c_{F},C_{F}>0$ depending on $\alpha$ and $\beta$ such that for any stopping time $\tau$ with respect to the filtration $\{\mathcal{F}_t\}$,
\begin{equation*}
c_{F}\Enum F\left(\log\left(1+\alpha\beta\tau\right)\right)
\leq \Enum F(R_\tau^*) \leq C_{F}\Enum F\left(\log\left(1+\alpha\beta\tau\right)\right).
\end{equation*}

\end{corollary}
\begin{proof}
Let $f(x) := F(\sqrt{x})$ for any $x\geq 0$. Since $F$ is a moderate function, it is easy to check that $f$ is also a moderate function. The desired result then follows from Theorem \ref{thmCIR}.
\end{proof}

\section{Complex Ornstein-Uhlenbeck processes}\label{thmou}
Thus far, the moderate maximal inequalities have been established for various one-dimensional diffusions. Interestingly, the results in this paper can also be used to establish the maximal inequalities for some high-dimensional processes. Here we consider the moderate maximal inequalities for the complex OU process, which is also an important kinetic model in statistical mechanics \cite{chen2016large}. We recall the following definition \cite{arato1962evaluation, chen2014eigenfunctions}.

\begin{definition}\label{defcomou}
Let $W = W^{(1)}+\mi{W^{(2)}}$ be a complex standard Brownian motion, where $W^{(1)}$ and $W^{(2)}$ are real standard Brownian motions with respect to the filtration $\{\mathcal{F}_t\}$. Then for any $\alpha = a+\mi b\in \mathbb{C}$ with $a>0$, the solution $Z = (Z_t)_{t\geq 0}$ of the following SDE
\begin{equation}\label{eq2.1}
\dif Z_t = -\alpha Z_t\dif t+ dW_t, \quad Z_0 = 0,
\end{equation}
is called a complex OU process starting from zero.
\end{definition}

The following theorem gives the moderate maximal inequalities for complex OU processes.

\begin{theorem}\label{Thm2.2}
Let $Z = (Z_t)_{t\geq 0}$ be the complex OU process solving \eqref{eq2.1}. Then for any moderate function $F$, there exist two constants $c_{\alpha,F},C_{\alpha,F}>0$ such that for any stopping time $\tau$ wth respect to the filtration $\{\mathcal{F}_t\}$,
\begin{equation}\label{eqcomplexou}
c_{\alpha,F}\Enum F\bigl(\log^{1/2}(1+2a \tau)\bigl)\leq \Enum F(Z^*_\tau)
\leq C_{\alpha,F}\Enum F\bigl(\log^{1/2}(1+2a \tau)\bigl).
\end{equation}
\end{theorem}

\begin{proof}
Let $Z_t = X_t+iY_t$ for any $t\geq 0$. It is easy to check that the pair $(X,Y)$ is the solution to the following two-dimensional diffusion:
$$\begin{cases}
dX_t=(-aX_t+bY_t)dt+W_{t}^{(1)},\\
dY_t=(-bX_t-aY_t)dt+W_{t}^{(2)}.
\end{cases}$$
By It\^{o}'s formula,  we have
\begin{align*}
|Z_t|^2=X_t^2+Y_t^2=&\int_0^t 2X_sdX_s+\int_0^t 2Y_sdY_s+\int_0^t d[X,X]_s+\int_0^t d[Y,Y]_s\\
=&\int_0^t(2-2aX_s^2-2aY_s^2)ds+2\int_0^t X_sdW_{s}^{(1)}+2\int_0^t Y_sdW_{s}^{(2)}\\
=&\int_0^t(2-2a|Z_s|^2)ds+2\int_0^t \sqrt{|Z_s|^2}dB_s,
\end{align*}
where the process $B = (B_t)_{t\geq 0}$ is defined as
\begin{equation*}
B_t = \int_0^t \frac{X_u}{\sqrt{|Z_u|^2}}dW_{u}^{(1)}+\frac{Y_u}{\sqrt{|Z_u|^2}}dW_{u}^{(2)}.
\end{equation*}
Note that $B$ is a continuous local martingale starting from zero and $\langle B, B\rangle_s=s$. It follows from L$\acute{e}$vy's characterization theorem \cite[P150 Theorem 3.6]{revuz1999continuous} that $B$ is a standard Brownian motion with respect to the filtration $\{\mathcal{F}_t\}$. Hence $|Z|^2$ is a CIR process. Let $f(x) = F(\sqrt{x})$ for any $x\geq 0$. Since $F$ is a moderate function, it is easy to check that $f$ is also a moderate function. The desired results then follows from Theorem \ref{thmCIR}.
\end{proof}

\section{Conformal local martingales}
In fact, the moderate maximal inequalities studied above can be used to established two types of moderate maximal inequalities for conformal local martingales, which can be viewed as an extension of the classical BDG inequality. We first recall the following definition \cite[Chapter V, Definition 2.2]{revuz1999continuous}.

\begin{definition}\label{defcomou}
Let $M=X+\mi Y$ be a continuous complex local martingale, i.e. $X$ and $Y$ are two real continuous local martingales. Then $M$ is called a conformal local martingale if
\begin{equation*}
[M,M] =[X,X]-[Y,Y]+2\mi [X,Y] = 0,
\end{equation*}
where $[X,Y]$ denotes the quadratic variation process between $X$ and $Y$.
\end{definition}

Before we focus on conformal local martingales, we establish two types of moderate maximal inequalities for complex Brownian motions, which are stated below. The idea of the following result is similar to Corollary 2.7 in \cite{graversen2000maximal} but with more complex calculations.

\begin{corollary} \label{coro2.3}
Let $W = W^{(1)}+\mi{W^{(2)}}$ be a complex standard Brownian motion with respect to the filtration $\{\mathcal{F}_t\}$. Then for any moderate function $F$, there exist two constants $c_F,C_F>0$ such that for any stopping time $\tau$ with respect to the filtration $\{\mathcal{F}_t\}$,
\begin{equation}\label{complexBM1}
c_F\Enum F(\sqrt{\tau})
\leq \Enum F\biggl(\max_{0\leq t\leq \tau}\abs{W_t}\biggr)
\leq C_F\Enum FF(\sqrt{\tau}),
\end{equation}
\begin{equation}\label{complexBM2}
c_F\Enum F\biggl(\log^{1/2}\bigl(1+\log(1+\tau)\bigr)\biggr)
\leq \Enum F\biggl(\max_{0\leq t\leq \tau}\frac{\abs{W_t}}{\sqrt{1+t}}\biggr)
\leq C_F\Enum F\biggl(\log^{1/2}\bigl(1+\log(1+\tau)\bigr)\biggr).
\end{equation}
\end{corollary}

\begin{proof}
Since $|W|$ is a two-dimensional Bessel process starting from zero, the inequalities \eqref{complexBM1} follow directly from Corollary \ref{bessel}.

On the other hand, for any $a>0$ and $b\in\Rnum$, let $\alpha = a+\mi b$. To proceed, we define
\begin{equation*}
\mathcal{G}_t = \mathcal{F}_{e^{2at}-1},\;\;\;
Z_t = \frac{1}{\sqrt{2a}}e^{-\alpha t}W_{e^{2at}-1},\;\;\;
\tilde{W}_t = Z_t+\alpha\int_0^tZ_sds,
\end{equation*}
for any $t\geq 0$. Note that $B_t^{(1)} = W^{(1)}_{e^{2at}-1}$ and $B_t^{(2)} = W^{(2)}_{e^{2at}-1}$ are continuous martingales with respect to the filtration $\{\mathcal{G}_t\}$. Moreover, it is easy to see that
\begin{equation*}
\sqrt{2a}Z_t = e^{-at}\left(\cos(bt)B_t^{(1)}+\sin(bt) B_t^{(2)}\right)
+\mi e^{-at}\left(-\sin(bt)B_t^{(1)}+\cos(bt)B_t^{(2)}\right).
\end{equation*}
By It\^{o}'s formula, we have
\begin{align*}
\sqrt{2a}dZ_t =&\; -e^{-at}\left(a\cos(bt)B_t^{(1)}+a\sin(bt)B_t^{(2)}+b\sin(bt)B_t^{(1)}-b\cos(bt)B_t^{(2)}\right)dt\\
&\; +e^{-at}\cos(bt)dB_t^{(1)}+e^{-at}\sin(bt)dB_t^{(2)}\\
&\; +\mi e^{-at}\left(a\sin(bt)B_t^{(1)}-a\cos(bt)B_t^{(2)}-b\cos(bt)B_t^{(1)}-b\sin(bt)B_t^{(2)}\right)dt\\
&\; -\mi e^{-at}\sin(bt)dB_t^{(1)}+\mi e^{-at}\cos(bt)dB_t^{(2)}\\
=&\; -\alpha\sqrt{2a}Z_tdt+e^{-at}\cos(bt)dB_t^{(1)}+e^{-at}\sin(bt)dB_t^{(2)}\\
&\; -\mi e^{-at}\sin(bt)dB_t^{(1)}+\mi e^{-at}\cos(bt)dB_t^{(2)}.
\end{align*}
Hence we obtain
\begin{align*}
\sqrt{2a}d\tilde{W}_t =&\; \sqrt{2a}dZ_t+\alpha\sqrt{2a}Z_tdt\\
=&\; e^{-at}\cos(bt)dB_t^{(1)}+e^{-at}\sin(bt)dB_t^{(2)}
-\mi e^{-at}\sin(bs)dB_t^{(1)}+\mi e^{-at}\cos(bs)dB_t^{(2)}.
\end{align*}
This shows that $\tilde{W}_t = \tilde{W}_t^{(1)}+\mi\tilde{W}_t^{(2)}$, where
\begin{gather*}
\tilde{W}_t^{(1)} = \frac{1}{\sqrt{2a}}\int_0^te^{-as}\cos(bs)dB_s^{(1)}
+\frac{1}{\sqrt{2a}}\int_0^te^{-as}\sin(bs)dB_s^{(2)},\\
\tilde{W}_t^{(2)} = \frac{1}{\sqrt{2a}}\int_0^te^{-as}\cos(bs)dB_s^{(2)}
-\frac{1}{\sqrt{2a}}\int_0^te^{-as}\sin(bs)dB_s^{(1)}.
\end{gather*}
Moreover, it is easy to check that
\begin{align*}
[\tilde{W}^{(1)},\tilde{W}^{(1)}]_t
=&\; \frac{1}{2a}\int_0^te^{-2as}\cos^2(bs)d[B^{(1)},B^{(1)}]_s+\frac{1}{2a}\int_0^te^{-2as}\sin^2 (bs)d[B^{(2)},B^{(2)}]_s\\
=&\; \frac{1}{2a}\int_0^te^{-2as}d(e^{2as}-1) = t.
\end{align*}
Similarly, we can prove that $[\tilde{W}^{(2)},\tilde{W}^{(2)}]_t = t$ and $[\tilde{W}^{(1)},\tilde{W}^{(2)}]_t = 0$.  This shows that $W$ is complex Brownian motion with respect to $\{\mathcal{G}_t\}$ and thus $Z$ is a complex OU process. Let $H(t)=e^{2at}-1$ for any $t\geq 0$. Note that $\tau$ is a stopping time with respect to $\{\mathcal{F}_t\}$ if and only if $H^{-1}(\tau) = \log(1+\tau)/(2a)$ is a stopping time of
$\{\mathcal{G}_t\}$. Note that
\begin{equation*}
\sqrt{2a}|Z_t| = e^{-at}|W_{e^{2at}-1}| = \frac{|W_{H(t)}|}{\sqrt{H(t)+1}}.
\end{equation*}
This shows that
\begin{equation*}
\max_{0\leq t\leq \tau}\frac{\abs{W_t}}{\sqrt{1+t}} = \sqrt{2a}\max_{0\leq t\leq H^{-1}(\tau)}|Z_t|.
\end{equation*}
Thus it follows from Theorem \ref{Thm2.2} that
\begin{equation*}
\begin{split}
\Enum\left[\sup_{0\leq t\leq\tau}f\left(\frac{1}{\sqrt{2a}}\frac{|W_t|}{\sqrt{1+t}}\right)\right]
&\sim \Enum f\left(\log^{1/2}\left(1+2aH^{-1}(\tau)\right)\right)\\
&= \Enum f\left(\log^{1/2}\left(1+\log(1+\tau)\right)\right),
\end{split}
\end{equation*}
where $x\sim y$ means that there exists two constants $c,C>0$  such that $cx\leq y\leq Cx$. The desired result then follows from the definition of moderate functions.
\end{proof}

Since any conformal local martingale is a time change of the complex Brownian motion, the above corollary implies the moderate maximal inequalities for conformal local martingales.

\begin{corollary}
Let $M = X+\mi Y$ be a conformal local martingale with respect to the filtration $\{\mathcal{F}_t\}$ starting from zero. Then for any moderate function $F$, there exist two constants $c_F,C_F>0$ such that for any  stopping time $\tau$ of the filtration $\{\mathcal{F}_t\}$,
\begin{equation}\label{mtgl1}
c_F\Enum F\left(\sqrt{[X,X]_\tau}\right)
\leq \Enum F\biggl(\max_{0\leq t\leq \tau}\abs{M_t}\biggr)
\leq C_F\Enum F\left(\sqrt{[X,X]_\tau}\right),
\end{equation}
\begin{equation}\label{mtgl2}
c_F\Enum F\left(g\left([X,X]_\tau\right)\right)
\leq \Enum F\biggl(\max_{0\leq t\leq \tau}\frac{\abs{M_t}}{\sqrt{1+[X,X]_t}}\biggr)
\leq C_F\Enum F\left(g\left([X,X]_\tau\right)\right).
\end{equation}
where $g(t) = \log^{1/2}(1+\log(1+t))$ for $t\geq 0$.
\end{corollary}

\begin{proof}
Since $M$ is a conformal local martingale with $M_0 = 0$, there exists a complex standard Brownian motion $W$ such that $M_t = W_{[X,X]_t}$ \cite[Chapter V, Theorems 2.4]{revuz1999continuous}. The desired result then follows directly from Corollary \ref{coro2.3}.
\end{proof}

For any conformal local martingale $M$, \eqref{mtgl1} shows that the maximum process of $|M|^p$ on average behaves as $[X,X]^{p/2}$ for any $p>0$. Furthermore, \eqref{mtgl2} shows that the maximum process of $|M|^p$, normalized by $(1+[X,X])^{p/2}$, on average behaves as $\log^{p/2}(1+\log(1+[X,X]))$ for any $p>0$. The relationship between these two results is rather similar to that between the law of large numbers and the central limit theorem.

\section*{Acknowledgements}
X. Chen is supported by National Natural Science Foundation of China (NSFC) with grant No. 11701483. Y. Chen is supported by NSFC with grant No. 11961033. C. Jia acknowledges support from the NSAF grant in NSFC with grant No. U1930402.

\setlength{\bibsep}{5pt}
\small\bibliographystyle{nature}

\begin{thebibliography}{28}
\expandafter\ifx\csname natexlab\endcsname\relax\def\natexlab#1{#1}\fi
\expandafter\ifx\csname url\endcsname\relax
  \def\url#1{\texttt{#1}}\fi
\expandafter\ifx\csname urlprefix\endcsname\relax\def\urlprefix{URL }\fi

\bibitem[{Revuz \& Yor(1999)}]{revuz1999continuous}
Revuz, D. \& Yor, M.
\newblock \emph{Continuous Martingales and Brownian motion} (Springer, 1999).

\bibitem[{Dubins \emph{et~al.}(1994)Dubins, Shepp \&
  Shiryaev}]{dubins1994optimal}
Dubins, L.~E., Shepp, L.~A. \& Shiryaev, A.~N.
\newblock Optimal stopping rules and maximal inequalities for Bessel processes.
\newblock \emph{Theory of Probability \& Its Applications} \textbf{38},
  226--261 (1994).

\bibitem[{Graversen \& Peskir(1998{\natexlab{a}})}]{graversen1998maximal}
Graversen, S. \& Peskir, G.
\newblock Maximal inequalities for Bessel processes.
\newblock \emph{J. Inequal. Appl.} \textbf{2}, 99--119 (1998{\natexlab{a}}).

\bibitem[{Graversen \& Peskir(1998{\natexlab{b}})}]{graversen1998optimal}
Graversen, S.~E. \& Peskir, G.
\newblock Optimal stopping and maximal inequalities for geometric Brownian
  motion.
\newblock \emph{J. Appl. Probab.} \textbf{35}, 856--872 (1998{\natexlab{b}}).

\bibitem[{Graversen \& Peskir(2000)}]{graversen2000maximal}
Graversen, S.~E. \& Peskir, G.
\newblock Maximal inequalities for the Ornstein-Uhlenbeck process.
\newblock \emph{P. Am. Math. Soc.} 3035--3041 (2000).

\bibitem[{Peskir(2001)}]{peskir2001bounding}
Peskir, G.
\newblock Bounding the maximal height of a diffusion by the time elapsed.
\newblock \emph{J. Theor. Probab.} \textbf{14}, 845--855 (2001).

\bibitem[{Botnikov(2006)}]{botnikov2006davis}
Botnikov, Y.~L.
\newblock Davis-type inequalities for some diffusion processes.
\newblock \emph{Journal of Mathematical Sciences} \textbf{137}, 4502--4509
  (2006).

\bibitem[{Lyulko \& Shiryaev(2014)}]{lyulko2014sharp}
Lyulko, Y.~A. \& Shiryaev, A.~N.
\newblock Sharp maximal inequalities for stochastic processes.
\newblock \emph{Proceedings of the Steklov Institute of Mathematics}
  \textbf{287}, 155--173 (2014).

\bibitem[{Yan \& Zhu(2004)}]{yan2004ratio}
Yan, L. \& Zhu, B.
\newblock A ratio inequality for Bessel processes.
\newblock \emph{Statistics \& probability letters} \textbf{66}, 35--44 (2004).

\bibitem[{Yan \& Li(2004)}]{yan2004maximal}
Yan, L. \& Li, Y.
\newblock Maximal inequalities for CIR processes.
\newblock \emph{Letters in Mathematical Physics} \textbf{67}, 111--124 (2004).

\bibitem[{Yan \emph{et~al.}(2005)Yan, Lu \& Xu}]{yan2005lp}
Yan, L., Lu, L. \& Xu, Z.
\newblock {$L^p$ estimates on a time-inhomogeneous diffusion process}.
\newblock \emph{J. Math. Phys.} \textbf{46}, 3513 (2005).

\bibitem[{Yan \& Zhu(2005)}]{yan2005lpestimates}
Yan, L. \& Zhu, B.
\newblock {$L^p$-estimates on diffusion processes}.
\newblock \emph{J. Math. Anal. Appl.} \textbf{303}, 418--435 (2005).

\bibitem[{Chen \& Jia(2017)}]{chen2017identification}
Chen, X. \& Jia, C.
\newblock Identification of unstable fixed points for randomly perturbed
  dynamical systems with multistability.
\newblock \emph{J. Math. Anal. Appl.} \textbf{446}, 521--545 (2017).

\bibitem[{Shen \emph{et~al.}(2019)Shen, Xu \& Ren}]{shen2019some}
Shen, J., Xu, X. \& Ren, Y.
\newblock Some improvements on the lp inequalities for diffusion processes.
\newblock \emph{Journal of Mathematical Inequalities} \textbf{13}, 1057--1069
  (2019).

\bibitem[{Jia \& Zhao(2020)}]{jia2020moderate}
Jia, C. \& Zhao, G.
\newblock Moderate maximal inequalities for the Ornstein-Uhlenbeck process.
\newblock \emph{P. Am. Math. Soc.} \textbf{148}, 3607--3615 (2020).

\bibitem[{Jia(2019)}]{jia2019sharp}
Jia, C.
\newblock Sharp moderate maximal inequalities for upward skip-free Markov
  chains.
\newblock \emph{J. Theor. Probab.} \textbf{32}, 1382--1398 (2019).

\bibitem[{Burkholder(1973)}]{burkholder1973distribution}
Burkholder, D.~L.
\newblock Distribution function inequalities for martingales.
\newblock \emph{Ann. Probab.} 19--42 (1973).

\bibitem[{Ornstein \& Uhlenbeck(1930)}]{ornstein1930on}
Ornstein, L.~S. \& Uhlenbeck, G.~E.
\newblock On the Theory of the Brownian Motion.
\newblock \emph{Phys. Rev.} \textbf{36}, 823--841 (1930).

\bibitem[{Chen \& Jia(2020)}]{chen2020mathematical}
Chen, X. \& Jia, C.
\newblock Mathematical foundation of nonequilibrium fluctuation--dissipation
  theorems for inhomogeneous diffusion processes with unbounded coefficients.
\newblock \emph{Stoch. Proc. Appl.} \textbf{130}, 171--202 (2020).

\bibitem[{Graversen \& Shiryaev(2000)}]{graversen2000extension}
Graversen, S.~E. \& Shiryaev, A.~N.
\newblock An extension of P. Levy's distributional properties to the case of a
  Brownian motion with drift.
\newblock \emph{Bernoulli} 615--620 (2000).

\bibitem[{Peskir \& Shiryaev(1999)}]{peskir1999maximal}
Peskir, G. \& Shiryaev, A.~N.
\newblock \emph{Maximal inequalities for reflected Brownian motion with drift}
  (University of Aarhus. Department of Theoretical Statistics, 1999).

\bibitem[{Ikeda(1989)}]{Ikeda}
Ikeda, Nobuyuki;~Watanabe, S.
\newblock \emph{Stochastic differential equations and diffusion processes}
  (North-Holland Publishing Co., 1989).

\bibitem[{Rogers \& Williams(2000)}]{rogers2000diffusions2}
Rogers, L. C.~G. \& Williams, D.
\newblock \emph{{Diffusions, Markov Processes, and Martingales: Volume 2, Ito
  Calculus}} (Cambridge University Press, Cambridge, 2000).

\bibitem[{Cox \emph{et~al.}(1985)Cox, Ingersoll~Jr \& Ross}]{cox1985theory}
Cox, J.~C., Ingersoll~Jr, J.~E. \& Ross, S.~A.
\newblock A Theory of the Term Structure of Interest Rates.
\newblock \emph{Econometrica} \textbf{53}, 385--408 (1985).

\bibitem[{G{\"o}ing-Jaeschke \& Yor(2003)}]{going2003survey}
G{\"o}ing-Jaeschke, A. \& Yor, M.
\newblock A survey and some generalizations of Bessel processes.
\newblock \emph{Bernoulli} \textbf{9}, 313--349 (2003).

\bibitem[{Chen \emph{et~al.}(2016)Chen, Ge, Xiong \& Xu}]{chen2016large}
Chen, Y., Ge, H., Xiong, J. \& Xu, L.
\newblock The large deviation principle and steady-state fluctuation theorem
  for the entropy production rate of a stochastic process in magnetic fields.
\newblock \emph{J. Math. Phys.} \textbf{57}, 073302 (2016).

\bibitem[{Arat{\'o} \emph{et~al.}(1962)Arat{\'o}, Kolmogorov \&
  Sinai}]{arato1962evaluation}
Arat{\'o}, M., Kolmogorov, A.~N. \& Sinai, Y.~G.
\newblock Evaluation of the parameters of a complex stationary Gauss-Markov
  process.
\newblock \emph{Doklady Akademii Nauk SSSR} \textbf{146}, 747--750 (1962).

\bibitem[{Chen \& Liu(2014)}]{chen2014eigenfunctions}
Chen, Y. \& Liu, Y.
\newblock On the eigenfunctions of the complex Ornstein--Uhlenbeck operators.
\newblock \emph{Kyoto Journal of Mathematics} \textbf{54}, 577--596 (2014).

\end{thebibliography}

\end{document}